\documentclass{tsp-short}
\usepackage{amsfonts}
\usepackage{amssymb}
\usepackage{newlfont}
\usepackage{amsthm}
\usepackage{euscript}
\usepackage{amsmath}
\usepackage{graphicx}
\usepackage{dsfont}
\usepackage{enumerate}

\newcommand{\Var}{\mathop \mathrm{Var}}
\newcommand{\supp}{\mathop \mathrm{supp}}
\newtheorem{Proposition}{Proposition}
\theoremstyle{definition}

\newtheorem{Lemma}{Lemma}
\newtheorem{Theorem}{Theorem}

\theoremstyle{Condition}

\theoremstyle{remark}

\newtheorem{Remark}{Remark}

\newcommand{\veps}{\varepsilon}
\newcommand{\vp}{\varphi}

\newcommand{\aph}{\alpha}

\begin{document}

\title[On properties of a flow with discontinuous drift]{On properties of a flow generated by an SDE with discontinuous drift}

\author{Olga V. Aryasova$^1$}
\address{Institute of Geophysics, National Academy of Sciences of Ukraine,
Palladin pr. 32, 03680, Kiev-142, Ukraine}
\email{oaryasova@mail.ru}
\author{Andrey Yu. Pilipenko$^{1,2}$}
\address{Institute of Mathematics,  National Academy of Sciences of
Ukraine, Tereshchenkivska str. 3, 01601, Kiev, Ukraine}
\email{apilip@imath.kiev.ua}

\footnotetext[1]{Research is partially supported by the Grant of the President of Ukraine ¹ F47/457-2012.}

\footnotetext[2]{Research is partially supported by State fund for fundamental researches of Ukraine and the Russian foundation for basic researches, Grant F40.1/023.}

 \dedicatory{}


\begin{abstract}
We consider a stochastic flow on $\mathds{R}$ generated by an SDE with its drift being a function of bounded variation. We show that the flow is differentiable with respect to the initial conditions. Asymptotic  properties of the flow are studied.
\end{abstract}
\maketitle \thispagestyle{empty}
\section*{Introduction}
Consider an SDE of the form
\begin{equation}\label{eq}
\left\{
\begin{aligned}
d\vp_t(x)&=\aph(\vp_t(x))dt+\sigma(\vp_t(x))dw(t),\\
\vp_0(x)&=x,\\
\end{aligned}\right.
\end{equation}
where $x\in\mathbb{R}, \ (w(t))_{t\geq0}$ is a one-dimensional Wiener process.

It is well known (cf. \cite{Kunita90}) that if the coefficients of
(\ref{eq}) are continuously differentiable and the derivatives are bounded and H{\"{o}}lder
continuous then there exists a flow of diffeomorphisms for equation
(\ref{eq}). Under the condition of  Lipschitz continuity of the
coefficients it was shown the existence of a flow of
homeomorphysms (ibid.). Moreover, in the latter situation Bouleau and
Hirsch \cite{Bouleau+91} established  the differentiability of the
flow in generalized sense. Recently, the essential improvement of the results was obtained by Flandoli et al. \cite{Flandoli+10}. They
proved the existence of a flow of diffeomorphysms in the case of a
smooth non-degenerate noise and a possibly unbounded H\"{o}lder
continuous drift term.

An SDE with bounded variation drift and $\sigma\equiv 1$ was
treated by Attanasio \cite{Attanasio10}, who  stated the existence
of stochastic flow of class $C^{1,\varepsilon}, \
\varepsilon<1/2,$ under the assumption about boundedness of the
positive or the negative part of the distributional derivative of
$\aph.$ We consider equation (\ref{eq}) with $\sigma\equiv 1$ and
$\aph$ being a function of bounded variation. We have not
additional assumptions about boundedness of the derivative.
Besides our method is different from their one.

Note that sometimes the strong solution may exist even if $\alpha$
is a measure. However, in this case the flow may be
discontinuous in $x$. For example, if $\aph(x)=\beta\delta_0(x), \
\sigma\equiv1,$ where $\beta\in[-1,1], \ \delta_0$ is a  Dirac
delta function at zero, then the corresponding strong solution of
(\ref{eq}) exists and it is a skew Brownion motion
\cite{Harrison+81} but the flow is discontinuous and coalescent (see Barlow et al.
\cite{Barlow+01} and Burdzy and Kaspi \cite{Burdzy+04}).

\section{The main results}
Consider an SDE
\begin{equation}\label{main_eq}
\left\{
\begin{aligned}
d\vp_t(x)&=\aph(\vp_t(x))dt+dw(t),\\
\vp_0(x)&=x,\\
\end{aligned}\right.
\end{equation}
where $x\in\mathbb{R}, \ \alpha$ is a function on $\mathds{R},$
$(w(t))_{t\geq0}$ is a one-dimensional Wiener process.

Later on the function $\aph$ will be assumed to satisfy some of
the following conditions.
\begin{enumerate}[(A)]
\item $\aph$ has bounded variation on each compact subset of $\mathds{R}$;\label{Condition_A}
\vskip 5 pt
\item for all $x\in\mathds{R}$\label{Condition_B}
$$
|\aph(x)|^2\leq C(1+|x|^2);
$$
\vskip 5 pt
\item $\aph$ is a function of bounded variation on
$\mathds{R}$;\label{BVonR}
\vskip 5 pt
\item there exist $a<0, b>0$
such that \label{limONinfty}
$$
\begin{aligned}
\aph(x)\to a,& \ \ x\to+\infty,\\
\aph(x)\to b,& \ \ x\to-\infty.
\end{aligned}
$$
\vskip 5 pt
\end{enumerate}

Given $p\geq 1$, denote by $W_{p,loc}^1(\mathds{R})$ the set of
functions defined on $\mathds{R}$ that belong to the Sobolev space
$W_{p}^1([c,d])$ for all $\{c,d\}\subset\mathds{R}$,  $c<d$. The
results about differentiability and non-coalescence of the flow
generated by equation (\ref{main_eq}) is represented as the
following statement.
\begin{Theorem}\label{meeting} {\it  Let $\aph$ satisfy conditions (\ref{Condition_A}), (\ref{Condition_B}). Then
\begin{enumerate}[1)]
\item For each $x\in\mathds{R}$ there exists a unique strong solution to equation (\ref{main_eq}).
\item For all $t\geq0,$
$$
P\{\forall p\geq1\ : \ \varphi_t(\cdot)\in W_{p,loc}^1(\mathds{R}) \}=1.
$$
\item For $t\geq 0$ the Sobolev  derivative $\nabla\varphi_t(x)$ is of the form
\begin{equation}\label{Sobolev_der_expres}
{P}\left\{\nabla\varphi_t(x)=\exp\left\{\int_{-\infty}^{+\infty}L_z^{\varphi(x)}(t)d\aph(z)\right\}, \ x\in\mathds{R}\right\}=1.
\end{equation}
where $L_z^{\varphi(x)}(t)$ is a local time of the process
$(\varphi_s(x))_{s\in[0,t]}$ at the point $z$. \item For all
$\left\{x_1,x_2\right\}\subset\mathds{R}, \ x_1\neq x_2,$
$$
P\left\{\varphi_t(x_1)\neq\varphi_t(x_2), \ t\geq0\right\}=1.
$$
\end{enumerate}
}
\end{Theorem}
\vskip 10 pt
\begin{Remark}\label{remark_local_time}
We define the local time of the process $(\vp_t(x))_{t\geq0}$ at the point $y\in\mathds{R}$ by the formula
$$
L_y^{\varphi(x)}(t)=\lim_{\veps\downarrow0}\frac{1}{\veps}\int_0^t\mathds{1}_{[y,y+\veps)}(\varphi_s(x))ds, \ t\geq0.
$$
\end{Remark}
\vskip 10 pt

We prove the Theorem \ref{meeting} in two stage. At the first one
we consider $\aph$ having a compact support on $\mathds{R}$. In
Sections \ref{Approximation}-\ref{differential_properties} we
obtain auxiliary results for this stage of proof. The Theorem is
proved in Section 5.

In the next sections we analyze the asymptotic behavior of the flow as
$t\to \infty.$ To do this in Section 6 we find the stationary
distribution of the process solving (\ref{main_eq}) under
conditions (\ref{BVonR}), (\ref{limONinfty}). The main result
about asymptotic properties of the flow is represented in the
following Theorem, proof of which can be found in Section 7. In
Section 8 the example is represented.

\vskip 5 pt

\begin{Theorem} \label{asymptot} {\it Let $\aph$ satisfy conditions (\ref{BVonR}), (\ref{limONinfty}). Then
for all $\{x_1,x_2\}\subset\mathds{R}, \ x_1<x_2,$
\begin{equation*}
\frac{\ln(\varphi_t(x_2)-\varphi_t(x_1))}{t}\to
\int_{-\infty}^{+\infty}\left(-\int_z^{+\infty}\aph(y)dP_{stat}(y)\right)d\aph(z),
\ t\to\infty, \ \mbox{almost surely,}
\end{equation*}
where $P_{stat}$ is a stationary distribution of the process $(\vp_t(x))_{t\geq0}.$}
\end{Theorem}
\vskip 5 pt
\begin{Remark}
Under the conditions of Theorem the stationary distribution of the process $(\vp_t(x))_{t\geq0}$ does not depend on the starting point $x$.
\end{Remark}

\section{Approximation of the SDE  by SDEs with smooth coefficients}\label{Approximation}
Let $\aph$ be a function of bounded variation on $\mathds{R}$ such that it has a compact support. Then
for each $x\in\mathds{R}$ there exists a unique strong solution to (\ref{main_eq}) (cf. \cite{Zvonkin74}).

For $n\geq1,$ let $g_n$ be a continuously differentiable function on $\mathds{R}$ equal to zero out of $\left(-\frac{1}{n},\frac{1}{n}\right)$ and such that $g_n(x)\geq0, x\in\mathds{R},$ $\int_{\mathds{R}}g_n(z)dz=1$. Put, for $x\in\mathds{R},$
$$
\aph_n(x)=\int_{\mathds{R}} g_n(x-y)\aph(y)dy.
$$
Then $\aph_n(x)\to\aph(x)$ as $n\to\infty$ at all points of continuity of $\aph$.

For $n\geq1$, consider an SDE
\begin{equation}\label{main_eq_n}
\left\{
\begin{aligned}
d\vp^n_t(x)&=\aph_n(\vp^n_t(x))dt+dw(t),\\
\vp^n_0(x)&=x.\\
\end{aligned}\right.
\end{equation}

\begin{Remark}\label{support_phi_n}
There exists $S>0$ such that for all $n\geq1, z\in\mathds{R},\ |z|\geq S$, $\aph_n(z)=0.$
Besides,
$$
\sup_{x\in\mathds{R}}|\aph_n(x)|\leq\sup_{x\in\mathds{R}}|\aph(x)|, \ n\geq1.
$$
\end{Remark}

\begin{Remark}\label{variation_phi_n}
For each $n\geq1$, $\aph_n$ is a function of bounded variation on $\mathds{R}$, and
$$
\Var_{\mathds{R}} \aph_n\leq\Var_{\mathds{R}}\aph.
$$
\end{Remark}

\begin{Lemma}\label{converg_Lp} {\it For each $p\geq1$,
\begin{enumerate}[1)]
\item for all $t\geq0,$
$$
\sup_{x\in\mathds{R}}\left(\mathds{E}(|\vp_t^n(x)|^p+|\vp_t(x)|^p)\right)<\infty;
$$
\item for all   $x\in\mathds{R}, \ t\geq0,$
$$
\mathds{E}|\vp^n_t(x)-\vp_t(x)|^p\to 0 \ \mbox{as} \ n\to \infty.
$$
\end{enumerate} }
\end{Lemma}
\begin{proof}
The convergence almost surely can be shown by arguments similar to
that of McKean \cite{McKean69}, Ch.3.10a. The boundedness of the
coefficients of (\ref{main_eq_n}) guarantees the uniform
boundedness of the moments:
$$
\sup_{n,x}\mathds{E}|\vp_t^n(x)-x|^p<\infty.
$$
This and convergence almost surely imply  the statement of Lemma.
\end{proof}

\section{Local times}\label{local}
For each $x\in \mathds{R}, n\geq 1 $, the processes
$(\varphi_t(x))_{t\geq0}$ and $(\varphi_t^n(x))_{t\geq0}$ solving
equations (\ref{main_eq}) and (\ref{main_eq_n}) are continuous
semimartingales. Then, almost surely, there exist local times of
these processes defined by the formulas
$$
\begin{aligned}
L_y^{\varphi(x)}(t)&=\lim_{\veps\downarrow0}\frac{1}{\veps}\int_0^t\mathds{1}_{[y,y+\veps)}(\varphi_s(x))d\langle\varphi(x),\varphi(x)\rangle_s
=\lim_{\veps\downarrow0}\frac{1}{\veps}\int_0^t\mathds{1}_{[y,y+\veps)}(\varphi_s(x))ds,
\\
L_y^{\varphi^n(x)}(t)&=\lim_{\veps\downarrow0}\frac{1}{\veps}\int_0^t\mathds{1}_{[y,y+\veps)}(\varphi^n_s(x))d\langle\varphi^n(x),\varphi^n(x)
\rangle_s=\lim_{\veps\downarrow0}\frac{1}{\veps}\int_0^t\mathds{1}_{[y,y+\veps)}(\varphi^n_s(x))ds.
\end{aligned}
$$
\begin{Remark} It follows from the definition that the local times is measurable with respect to the triple $(t,x,y),$ $t>0, \ x\in\mathds{R}$, \ $y\in\mathds{R}.$
\end{Remark}
\begin{Remark} \label{Remark_contin_mod} The family $L^{\vp(x)}, L^{\vp^n(x)}$ may be chosen such
that the maps $(t,y)\rightarrow L_y^{\varphi(x)}(t)$,
$(t,y)\rightarrow L_y^{\varphi^n(x)}(t)$ are continuous in $t$ and
c\'adl\'ag in $y$ (cf. \cite{Revuz+99}, Ch.VI). Further we
consider such modifications.
\end{Remark}

In this section we prove the convergence in square mean and
tightness of the sequence of the local times
$\{L_y^{\varphi^n(x)}(t)-L_y^{\varphi(x)}(t): \ n\geq 1\}.$

\begin{Lemma}\label{loc_time_conv_L2}
{\it For all $t\geq0, \ \{x, y\}\subset\mathds{R}$,
$$
\mathds{E}|L_y^{\varphi^n(x)}(t)-L_y^{\varphi(x)}(t)|^2\to 0 \ \
\mbox{as} \ n\to\infty.
$$ }
\end{Lemma}
\begin{proof}
By Tanaka's formula (see \cite{Revuz+99}, p. 223)
\begin{multline}\label{local_phi_n}
L_y^{\varphi^n(x)}(t)=(\varphi^n_t(x)-y)^+-(x-y)^+-\int_0^t\mathds{1}_{(y,\infty)}(\varphi_s^n(x))dw(s)\\-\int_0^t\mathds{1}_{(y,\infty)}(\varphi_s^n(x))\aph_n(\varphi_s^n(x))ds.
\end{multline}
\begin{multline}\label{local_phi}
L_y^{\varphi(x)}(t)=(\varphi_t(x)-y)^+-(x-y)^+-\int_0^t\mathds{1}_{(y,\infty)}(\varphi_s(x))dw(s)\\-\int_0^t\mathds{1}_{(y,\infty)}(\varphi_s(x))\aph(\varphi_s(x))ds.
\end{multline}
Then
$$
\mathds{E}\left(L_y^{\varphi^n(x)}(t)-L_y^{\varphi(x)}(t)\right)^2\leq
K(I+II+III),
$$
where $K$ is a constant,
\begin{eqnarray*}
I&=&\mathds{E}\left((\varphi^n_t(x)-y)^+-(\varphi_t(x)-y)^+\right)^2,\\
II&=&\mathds{E}\left(\int_0^t\mathds{1}_{(y,\infty)}(\varphi_s^n(x))dw(s)-\int_0^t\mathds{1}_{(y,\infty)}(\varphi_s(x))dw(s)\right)^2,\\
III&=&\mathds{E}\left(\int_0^t\mathds{1}_{(y,\infty)}(\varphi_s^n(x))\aph_n(\varphi_s^n(x))ds-\int_0^t\mathds{1}_{(y,\infty)}(\varphi_s(x))\aph(\varphi_s(x))ds\right)^2.
\end{eqnarray*}
For $I$ the convergence  follows from Lemma \ref{converg_Lp}.

To prove the convergence of $II$ and $III$ to $0$ we need the following statement.
\begin{Proposition}\label{Theor_Propos} {\it Let $\{\xi_n: n\geq0\}$ be a sequence of random variables. Assume that for any $n\geq 1$ the distribution of $\xi_n$ is absolutely continuous w.r.t. a probability measure $\nu$. Denote the corresponding density by $q_n$. Let $\{f_n: \ n\geq0\}$ be a sequence of measurable functions. Suppose that the following conditions hold:
\begin{enumerate}[1)]
\item $\xi_n\to\xi_0, \ n\to\infty$ in probability;
\item $f_n\to f_0, \ n\to\infty$ in measure $\nu$;
\item the sequence of densities $\{q_n: \ n\geq1\}$ is uniformly integrable w.r.t. measure $\nu$.
\end{enumerate}
Then $f_n(\xi)\to f_0(\xi_0), \ n\to\infty,$ in probability.}
\end{Proposition}
\begin{proof}
The proof is similar to \cite{Kulik+00}, Lemma 2.
\end{proof}
   According to the Lebesgue dominated convergence theorem,  to prove that $III\to 0, \ n\to\infty,$ it is enough to show that
   \begin{equation}\label{proposition}
\mathds{1}_{(y,\infty)}(\varphi_s^n(x))\aph_n(\varphi_s^n(x))\to\mathds{1}_{(y,\infty)}(\varphi_s(x))\aph(\varphi_s(x)), \ n\to \infty,  \ \mbox{in probability}.
   \end{equation}
   Apply the Proposition  \ref{Theor_Propos} in which we put $\xi_n=\varphi_s^n(x).$ Let $g_n(t,x,y), \ t\geq0, x\in\mathds{R}, y\in\mathds{R},$ be the transition probability density of the process $(\varphi_t^n(x))_{t\geq0}$. The density satisfies the inequality (cf. \cite{Portenko90}, Lemma 2.10)
\begin{equation}\label{denn}
g_n(t,x,y)\leq K\frac{1}{\sqrt{t}}e^{-\mu\frac{(y-x)^2}{t}}
\end{equation}
in every domain of the form $t\in[0,T], \ x\in\mathds{R}, \
y\in\mathds{R}.$ Here $T>0,\ \mu\in(0,1/2),$ $K$ is a constant
that depends only on $T, \ \mu$ and $\sup_{n,x}|\aph_n(x)|.$
Put
$$
\rho(y)=C\exp\left\{-\mu\frac{(y-x)^2}{t}\right\},
$$
and
$$
\nu(dy)=\rho(y)dy,
$$
where $C=\sqrt{\mu/(\pi t)}.$ Then the distribution of $\xi_n$ is absolutely continuous w.r.t. $\nu$, and the corresponding Radon-Nikodim density is equal to
$$
q_n(t,x,y)=\frac{g_n(t,x,y)}{\rho(y)}.
$$
The sequence $\{q_n(t,x,y): \ n\geq1\}$ is uniformly bounded for fixed $t>0, \ x\in\mathds{R},$ and, consequently, uniformly integrable w.r.t. measure $\nu$. Thus by Proposition \ref{Theor_Propos} relation (\ref{proposition}) is justified. The convergence of $II$ can be shown analogously.
The Lemma is proved.
\end{proof}

\begin{Lemma} \label{tightness}
{\it Let $\{c,d\}\subset\mathds{R}, \ c<d.$ Then
\begin{enumerate}[1)]
\item For each pair $(t,x)\in[0,\infty)\times\mathds{R}$, the
local times $L_{y}^{\varphi(x)}(t),$ $L_{y}^{\varphi^n(x)}(t), \
n\geq1$, are continuous in $y$ on $[c,d]$. \item For each fixed
pair $(t,x), \ t\geq0, \ x\in\mathds{R}$, the family of random
elements $\{L_{\centerdot}^{\varphi^n(x)}(t)
-L_\centerdot^{\varphi(x)}(t): \ n\geq1\}$ is tight in $C([c,d]).$
\end{enumerate} }
\end{Lemma}
\begin{proof} We prove the Lemma for $c=-1, \ d=1.$ The case of
arbitrary $c, d$ can be treated similarly.

Put
$$
R_{y}^{x,n}(t)=L_y^{\varphi^n(x)}(t)-L_y^{\varphi(x)}(t).
$$
By virtue of \cite{Billingsley68}, Theorem 12.3  to prove the tightness it is enough to show that
\begin{enumerate}[1)]
\item the sequence $\{R^{x,n}_{0}(t): \ n\geq1\}$ is tight;
\item there exist $\gamma\geq0, \aph>1,$ and $K>0$, such that for all $\{y_1,y_2\}\subset[-1,1]$
\begin{equation}\label{Kolmogorov_ineq}
\mathds{E}|R^{x,n}_{y_2}(t)-R^{x,n}_{y_1}(t)|^\gamma\leq K|y_2-y_1|^\aph.
\end{equation}
\end{enumerate}
Besides, according to \cite{Billingsley68}, Th.12.4, inequality (\ref{Kolmogorov_ineq}) provides the continuity of $R^{x,n}_{y}(t)$ with respect to  $y$  on $[-1,1]$ for each pair $(t,x)$ and each $n\geq1$.

The first item follows from Lemma \ref{loc_time_conv_L2} since
the fact that $\mathds{E}(R^{x,n}_0(t))^2\to 0$ as $n\to\infty,$ implies $L_0^{\varphi^n(x)}(t)-L_0^{\varphi(x)}(t)\to 0$ in probability as $n\to\infty$. The convergence ensures the tightness.

The proof of the second item is standard enough. We give  necessary calculations though.
Assume that $y_1<y_2$ and represent $L_{y_2}^{\varphi^n(x)}(t)-L_{y_1}^{\varphi^n(x)}(t)$ in the form
\begin{equation}\label{tightness_summand}
L_{y_2}^{\varphi^n(x)}(t)-L_{y_1}^{\varphi^n(x)}(t)=I-II-III-IV,
\end{equation}
where
\begin{eqnarray*}
I&=&(\varphi_t^n(x)-y_2)^+-(\varphi_t^n(x)-y_1)^+,\\
II&=&(x-y_2)^+-(x-y_1)^+,\\
III&=&\int_0^t \mathds{1}_{(y_1,y_2]}(\varphi_s^n(x))dw(s),\\
IV&=&\int_0^t\mathds{1}_{(y_1,y_2]}(\varphi_s^n(x))\aph_n(\varphi_s^n(x))ds.
\end{eqnarray*}

It is easy to see that
\begin{equation}\label{I_Kolmogorov}
\mathds{E}I^2\leq(y_2-y_1)^2, \  \mathds{E}II^2\leq(y_2-y_1)^2,
\end{equation}
Making use of Burkholder's inequality (cf. \cite{Ikeda+81}, Ch.3, Th. 3.1) we obtain that for each fixed $T>0$,
\begin{multline*}
\mathds{E}\max_{0\leq t\leq T}III^4\leq C\mathds{E}\left(\int_0^T\mathds{1}_{(y_1,y_2]}(\varphi_s^n(x))ds\right)^2\\
\leq 2C\mathds{E}\left(\int_0^T\mathds{1}_{(y_1,y_2]}(\varphi_s^n(x))ds\int_s^T\mathds{1}_{(y_1,y_2]}(\varphi_u^n(x))du \right)\\
=2C\int_0^T ds\int_s^T\mathds{E}\left(\mathds{1}_{(y_1,y_2]}(\varphi_s^n(x))\mathds{1}_{(y_1,y_2]}(\varphi_u^n(x))\right)du\\
= 2C\int_0^T ds\int_s^T du\int_{y_1}^{y_2}dy\int_{y_1}^{y_2}g_n(s,x,y)g_n(u-s,y,z)dz,
\end{multline*}
where $g_n(t,x,y), \ t\geq0, \ x\in\mathds{R}, \ y\in\mathds{R},$ is the transition probability density of the process $(\varphi^n_t(x))_{t\geq0}$.
So we have (see (\ref{denn}))
\begin{multline}\label{II_Kolmogarov}
\mathds{E}\max_{0\leq t\leq T}III^4\leq 2CK^2\int_0^Tds\int_s^Tdu\int_{y_1}^{y_2}dy\int_{y_1}^{y_2}\frac{e^{-\mu\frac{(y-x)^2}{s}}}{\sqrt{s}}\frac{e^{-\mu\frac{(z-y)^2}{u-s}}}{\sqrt{u-s}}dz\\
\leq\tilde{K}(y_2-y_1)^2\int_0^t\int_s^t s^{-1/2}(u-s)^{-1/2}du\leq\tilde{K}(y_2-y_1)^2T.
\end{multline}
Here we denote by $\tilde{K}$ different constants.

Using H\"{o}lder inequality and Remark \ref{variation_phi_n} we obtain
\begin{multline*}
\mathds{E}IV^4\leq\mathds{E}\left[\left(\int_0^t\mathds{1}_{(y_1,y_2]}(\varphi_s^n(x))ds\right)^2\left(\int_0^t\aph_n^2(\varphi_s^n(x))ds\right)^2\right]\\
\leq\left(||\aph|| t\right)^2\mathds{E}\left(\int_0^t\mathds{1}_{(y_1,y_2]}(\varphi_s^n(x))ds\right)^2.
\end{multline*}
Then from estimate (\ref{II_Kolmogarov}) we obtain
\begin{equation}\label{III_Kolmogorov}
\mathds{E}IV^4\leq\tilde K t^3(y_2-y_1)^2.
\end{equation}
So each summand in the right-hand side of (\ref{tightness_summand}) satisfies the second condition of Theorem 12.3 of \cite{Billingsley68}. Then the left-hand side of (\ref{tightness_summand}) is continuous with respect to $y$ and  tight in $C([-1,1])$.
Note that the estimates similar to (\ref{I_Kolmogorov})--(\ref{III_Kolmogorov}) hold for the process $(\varphi_t(x))_{t\geq0}$. This fact guarantees the continuity of  $L_\centerdot^{\varphi(x)}(t)$ with respect to $y$ on $[-1,1]$.
So $\{L_\centerdot^{\varphi^n(x)}(t)-L_\centerdot^{\varphi(x)}(t): \ n\geq1\}$ is a tight sequence of random elements in $C([-1,1])$.
The Lemma is proved.
\end{proof}

\section{Differential properties of the flow $\varphi_t(x)$}\label{differential_properties}

Denote by $\psi_t^n(x)$ the derivative of the function $\varphi_t^n(x)$ with
respect to $x$, i.e.
$$
\psi_t^n(x)=\left(\varphi_t^n(x)\right)'_x.
$$
Then $\psi_t^n(x)$ is a solution to the following differential
equation
$$
d\psi_t^n(x)=\aph_n'(\varphi_t^n(x))\psi_t^n(x)dt.
$$
Solving this equation we get
\begin{equation}\label{derivative}
\psi_t^n(x)=\exp\left\{\int_0^t\aph_n'(\varphi_s^n(x))ds\right\}.
\end{equation}
\begin{Lemma}\label{limit_local_Theorem}
{\it For all $t\geq0, \ x\in\mathds{R},$
\begin{equation*}\label{limit_to_local}
\int_0^t\aph_n'(\varphi_s^n(x))ds\rightarrow\int_{-\infty}^{+\infty}L_z^{\varphi(x)}(t)d\aph(z), \ n\to\infty,
\end{equation*}
in probability.}
\end{Lemma}
\begin{proof}
For each pair $(t,x), \ t\geq 0, \ x\in\mathds{R},$ according to the occupation times formula (see \cite{Revuz+99}, Ch.VI, Corollary 1.6) we have, almost surely,
\begin{multline}\label{uuu}
\int_0^t\aph_n'(\varphi_s^n(x))ds=\int_{\mathds{R}}\aph_n'(z)L_z^{\varphi^n(x)}(t)dz\\
=\int_{\mathds{R}}\aph_n'(z)(L_z^{\varphi^n(x)}(t)-L_z^{\varphi(x)}(t))dz+\int_{\mathds{R}}\aph_n'(z)L_z^{\varphi(x)}(t)dz\\
=\int_{\mathds{R}}(L_z^{\varphi^n(x)}(t)-L_z^{\varphi(x)}(t))d\aph_n(z)+\int_{\mathds{R}}L_z^{\varphi(x)}(t)d\aph_n(z)=I+II.
\end{multline}
Remark \ref{support_phi_n} and the continuity of the processes
$(L_z^{\vp(x)}(t))_{t\geq0}$, $(L_z^{\vp^n(x)}(t))_{t\geq0}$ in
$z$ entail the existence of the integrals in the right-hand side
of (\ref{uuu}). Besides, this leads to the relation
\begin{equation*}
II\to\int_{-\infty}^{+\infty}L_z^{\varphi(x)}(t)d\aph(z), \  n\to\infty, \ \mbox{almost surely.}
\end{equation*}
To prove the Lemma it remains to show that
$$
I\to 0, \ n\to\infty,  \ \mbox{in probability}.
$$
Lemma \ref{tightness} together with Prokhorov's theorem (cf.
\cite{Billingsley68}, Th.6.1) show that the family
$\left\{L_\centerdot^{\varphi^n(x)}(t)-L_\centerdot^{\varphi(x)}(t):
\ n\geq1\right\}$ is relatively compact in $C([-S,S])$ (here $S$
is a constant defined in Remark \ref{support_phi_n}). By Lemma
\ref{loc_time_conv_L2}, all the finite-dimensional distributions
of $L_\centerdot^{\varphi^n(x)}(t)-L_\centerdot^{\varphi(x)}(t)$
converge to that of  the random element in $C([-S,S])$ identically
equal to $0$. Therefore the sequence of random elements
$\left\{L_\centerdot^{\varphi^n(x)}(t)-L_\centerdot^{\varphi(x)}(t):
\ n\geq1\right\}$ converge in distribution to $0$ in $C([-S,S])$
(see \cite{Billingsley68}, Theorem 8.1). Then for all
$\varepsilon>0,$
$$
P\left\{\sup_{y\in[-S,S]}\left|L_y^{\varphi^n(x)}(t)-L_y^{\varphi(x)}(t)\right|>\varepsilon\right\}\to 0, \ n\to\infty.
$$
We have (remind that for all $n\geq 1$, $\supp\alpha_n\in[-S,S]$)
\begin{multline*}
P\left\{\left|\int_\mathds{R}\left(L_z^{\varphi^n(x)}(t)-L_z^{\varphi(x)}(t)\right)d\aph_n(z)\right|>\varepsilon\right\}\\
\leq P\left\{\sup_{y\in[-S,S]}\left|L_y^{\varphi^n(x)}(t)-L_y^{\varphi(x)}(t)\right|\cdot\Var_\mathds{R}\aph_n>\varepsilon\right\}\\
\leq P\left\{\sup_{y\in[-S,S]}\left|L_y^{\varphi^n(x)}(t)-L_y^{\varphi(x)}(t)\right|>\frac{\varepsilon}{\Var_{\mathds{R}} \aph}\right\}\to
0, n\to\infty.
\end{multline*}
The assertion of the Lemma follows immediately.
\end{proof}
\section{Proof of Theorem \ref{meeting}}
\begin{proof}{\bf Stage 1.} Let $\aph$ be a function of bounded variation on $\mathds{R}$ having a compact support.

Lemma \ref{limit_local_Theorem} guarantees that for each $t\geq 0,
\ x\in\mathds{R},$
\begin{equation}\label{conv_in_prob}
\psi_t^n(x)=\exp\left\{\int_0^t\aph_n'(\varphi_s^n(x))ds\right\}
\to
\exp\left\{\int_{\mathds{R}}L_z^{\varphi(s)}(t)d\aph(z)\right\}=:\psi_t(x), \ n\to\infty,
\end{equation}
in probability.  Let us estimate the $p$th moment of the process
$(\psi_t^n(x))_{t\geq0}$.

For all $p\geq1, \ t\geq0, \ x\in\mathds{R}$ by occupation times
formula, we have
$$
\mathds{E}|\psi_t^n(x)|^p=\mathds{E}\exp\left\{p\int_0^t\aph_n'(\varphi_s^n(x))ds\right\}=\mathds{E}\exp\left\{p\int_{-\infty}^{+\infty}
L_z^{\vp^n(x)}(t)d\aph_n(z)\right\}.
$$
Let $p_1,\dots,p_4$ be such that $p_k>1, \ k=\overline{1,4},$ and $\sum_{k=1}^{4}\frac{1}{p_k}=1.$ Using H\"{o}lder's inequality and Tanaka's formula we get
$$
\mathds{E}|\psi_t^n(x)|^p\leq\prod_{k=1}^{4}\left(\mathds{E}\exp\left\{pp_k\int_{-\infty}^{+\infty}f_k(t,x,z)d\aph_n(z)\right\}\right)^{1/p_k},
$$
where
\begin{equation*}
\begin{aligned}
f_1(t,x,z)&=(\vp_t^n(x)-z)^+,\\
f_2(t,x,z)&=-(x-z)^+,\\
f_3(t,x,z)&=-\int_0^t\mathds{1}_{(z,+\infty)}(\vp_s^n(x))dw(s),\\
f_4(t,x,z)&=-\int_0^t\mathds{1}_{(z,+\infty)}(\vp_s^n(x))\aph_n(\vp_s^n(x))ds.
\end{aligned}
\end{equation*}
Then we have
\begin{multline}\label{unif_integr_i}
\mathds{E}\exp\left\{pp_1\int_{-\infty}^{+\infty}f_1(t,x,z)d\aph_n(z)\right\}\\
\leq\mathds{E}\exp\left\{pp_1\int_{-\infty}^{+\infty}\left(|x-z|+
\int_0^t|\aph_n(\vp_s(x))|ds+|w(t)|\right)d\aph_n(z)\right\}\\
\leq \exp\left\{Cpp_1\Var_\mathds{R}\aph+||\aph||t\right\}\mathds{E}e^{|w(t)|},
\end{multline}
\begin{equation}\label{unif_integr_ii}
\mathds{E}\exp\left\{pp_2\int_{-\infty}^{+\infty}f_2(t,x,z)d\aph_n(z)\right\}\leq \mathds{E}\exp\left\{Cpp_2\Var_{\mathds{R}}\aph\right\},
\end{equation}
where $C$ is some positive constant,
\begin{equation}\label{unif_integr_iv}
\mathds{E}\exp\left\{pp_4\int_{-\infty}^{+\infty}f_4(t,x,z)d\aph_n(z)\right\}\leq \mathds{E}\exp\left\{pp_4\Var_{\mathds{R}}\aph\cdot||\aph||\right\},
\end{equation}

Consider $f_3$. Using Jensen's inequality, we get
\begin{multline}\label{mmm}
\mathds{E}\exp\left\{pp_3\int_{-\infty}^{+\infty}f_3(t,x,z)d\aph_n(z)\right\}\\
\leq \frac{1}{\Var_{\mathds{R}}\aph_n}\int_{-\infty}^{+\infty}\mathds{E}\exp\left\{pp_3\Var_{\mathds{R}}\aph_n f_3(t,x,z)\right\}d\aph_n(z)\\
\leq\frac{1}{\Var_{\mathds{R}}\aph_n}\int_{-\infty}^{+\infty}\sup_{v\in\mathds{R}}\mathds{E}\exp\left\{pp_3\Var_{\mathds{R}}\aph_n f_3(t,x,v)\right\}d\aph_n(z)\\
=\sup_{v\in\mathds{R}}\mathds{E}\exp\left\{pp_3\Var_{\mathds{R}}\aph_n f_3(t,x,v)\right\}.
\end{multline}
Let $v$ be fixed.
By \cite{Ikeda+81}, Th. II.7.2$'$, for each pair $(x,v)$, the process \break $M_t(x,v):=-\int_0^t\mathds{1}_{(v,+\infty)}(\vp_s^n(x))dw(s)$, $t\geq0,$ is a local square integrable martingale that can be represented as follows
$$
M_t(x,v)=W^{x,v}(\tau_t(x,v)),
$$
where $(W^{x,v}(t))_{t\geq0}$ is a standard Wiener process, $\tau_t(x,v)=\int_0^t\mathds{1}_{(v,+\infty)}(\vp_s^n(x))ds$.

Note that for all $\{x,z\}\subset\mathds{R}$, $\tau_t(x,v)\leq t$. Then
\begin{equation*}
\mathds{E}\exp\left\{pp_3\Var_{\mathds{R}}\aph_n f_3(t,x,v)\right\}
\leq\mathds{E}\exp\left\{pp_3\Var_{\mathds{R}}\aph \sup_{s\in[0,t]}|W^{x,v}(s)|\right\}=C,
\end{equation*}
where $C$ is a constant independent of $x$ and $v$.
This and (\ref{mmm}) imply the estimate
\begin{equation}\label{unif_integr_iii}
\mathds{E}\exp\left\{pp_3\int_{-\infty}^{+\infty}f_3(t,x,z)d\aph_n(z)\right\}\leq C,
\end{equation}
where $C$ ia some constant.
Now the uniform boundedness of the $p$th moment follows from
inequalities (\ref{unif_integr_i})-(\ref{unif_integr_iii}). This and (\ref{conv_in_prob}) imply that
for all $t\geq0,  \ p\geq1$,
$$
\mathds{E}|\psi_t^n(x)-\psi_t(x)|^p\to 0, \
n\to\infty.
$$
Since
$$
\sup_{n,x}\mathds{E}\left(|\psi_t^n(x)|^p+|\psi_t(x)|^p\right)<\infty,
$$
by the dominated convergence theorem, we get the relation
$$
\mathds{E}\int_{c}^d\left|\psi_t^n(x)-\psi_t(x)\right|^pdx\to 0, \
n\to\infty,
$$
valid for all $\left\{c,d\right\}\in\mathds{R}, \ c<d, \ p\geq1.$ So there exists a subsequence $\left\{n_k: \ k\geq1\right\}$ such
that
\begin{equation*}
\int_{c}^d\left|\psi_t^{n_k}(x)-\psi_t(x)\right|^pdx\to 0 \ \mbox{a.s. as} \
n\to \infty.
\end{equation*}
Without loss of generality we can suppose that
\begin{equation}\label{Sobolev_der}
\int_{c}^d\left|\psi_t^{n}(x)-\psi_t(x)\right|^pdx\to 0 \ \mbox{a.s. as} \
n\to \infty,
\end{equation}
and (see Lemma \ref{converg_Lp})
\begin{equation}\label{Sobolev_der}
\int_{c}^d\left|\vp_t^{n}(x)-\vp_t(x)\right|^pdx\to 0 \ \mbox{a.s. as} \
n\to \infty.
\end{equation}

This imply that, almost surely, the
function $\varphi_t(x), \ t\geq0, \ x\in\mathds{R},$ has a weak derivative in the Sobolev sense with respect to $x$ in any interval $[c,d]$ (cf.
\cite{Nikolsliy75}, \S 19.5), and this derivative is equal to
$\psi_t(x)=\exp\left\{\int_{-\infty}^{+\infty}L_z^{\varphi(x)}(t)d\aph(z)\right\}, \ x\in[c,d]$.
Besides, for all $\left\{x_1,x_2\right\}\subset\mathds{R},$ $x_1< x_2,$ the
equality
\begin{equation}\label{Newton_Leib}
\varphi_t(x_2)-\varphi_t(x_1)=\int_{x_1}^{x_2}\psi_t(y)dy=\int_{x_1}^{x_2}\exp\left\{\int_{-\infty}^{+\infty}L_z^{\varphi(y)}(t)d\aph(z)\right\}dy
\end{equation}
holds true  almost surely. Note that generally the exceptional set depends on $t$.

Fix $T>0$. Since $L_z^{\vp(y)}(t)$ is continuous in $t$ and $z$  (see Remark \ref{Remark_contin_mod}), monotonic in $t$, and $\supp\aph\subset[-S,S]$, we have
\begin{multline*}
\forall y\in [x_1,x_2] \ \ P\left\{\int_{-\infty}^{+\infty}L_z^{\vp(y)}(t)d\aph(z)\geq -\sup_{z\in[-S,S]}L_z^{\vp(y)}(t)\cdot||\aph||\right.\\
\left.\geq -\sup_{z\in[-S,S]}L_z^{\vp(y)}(T)\cdot||\aph||>-\infty, \ t\in[0,T]\right\}=1.
\end{multline*}
Put $M_T(y)=\sup_{z\in[-S,S]}L_z^{\vp(y)}(T)\cdot||\aph||.$
Then by the continuity of $L_z^{\vp(y)}(t)$ in $t$ and Fubini's theorem,
$$
P\left\{\inf_{t\in[0,T]}\int_{-\infty}^{+\infty}L_z^{\vp(y)}(t)d\aph(z)\geq -M_T(y)>-\infty \ \mbox{for almost all} \ y\in[x_1,x_2]\right\}=1
$$
This implies that
for all $T>0,\ x_1<x_2,$
\begin{multline*}
P\left\{\inf_{t\in[0,T]}(\varphi_t(x_2)-\varphi_t(x_1))>0\right\}=P\left\{\inf_{t\in[0,T]}\int_{x_1}^{x_2}\exp\left\{\int_{-\infty}^{+\infty}L_z^{\varphi(y)}(t)d\aph(z)\right\}dy \right\}\geq\\
P\left\{ \inf_{t\in[0,T]}  \int_{x_1}^{x_2}\exp\left\{-M_{T}(y)\right\}dy>0\right\}=1.
\end{multline*}
Passing to the limit as $T$ tends to $+\infty$, we arrive at the relation
$$
P\left\{\varphi_t(x_2)-\varphi_t(x_1)>0, \ t\geq0\right\}=1.
$$

{\bf Stage 2.} Let $\aph$ be an arbitrary function on $\mathds{R}$ satisfying conditions (\ref{Condition_A}), (\ref{Condition_B}).
For $n\geq 1,$ let $h_n$ be a smooth function on $\mathds{R}$ such that $0\leq h_n(x)\leq 1, \ x\in\mathds{R};$ $h_n(x)=1, \ x\in[-n,n];$  $h_n(x)=0, \ |x|>n+1.$ Put
$$
\aph_n(x)=\aph(x)h_n(x), x\in\mathds{R}.
$$
Suppose $(\vp_t^n(x))_{t\geq0}$ is a solution of equation (\ref{main_eq_n}).
Put
$\tau_n=\sup\left\{t: \ \sup_{0\leq s\leq t}|\vp_t^n(x)|\leq n\right\}$. As $\aph(x)=\aph_n(x)$ on $[-n,n]$, we have $\vp_t(x)=\vp_t^n(x)$ on $[0,\tau_t]$ almost surely. To prove the existence and uniqueness of a strong solution to equation (\ref{main_eq}) we need to show that $\tau_n\to+\infty, \ n\to\infty,$ almost surely. By Chebyshev's inequality  and condition (\ref{Condition_B}) for  $T>0$,
\begin{multline*}
P\{\tau_n<T\}=P\{\sup_{0\leq t\leq T}|\vp_t^n(x)|>n\}\leq \frac{1}{n^2}\mathds{E}\left(\sup_{0\leq t\leq T}(\vp_t^n(x))^2\right)\\
\leq \frac{C}{n^2}\left((\vp_t^n(0))+T\mathds{E}\sup_{0\leq t\leq T}\int_0^t(1+(\vp_t^n(x))^2)ds+T\right)\\
\leq \frac{C}{n^2}\left(K(1+T+T^2)+T\int_0^T\mathds{E}\sup_{0\leq s\leq T}(\vp_s^n(x))^2ds\right),
\end{multline*}
where $C,K$ are some positive constants.
The Gronwall-Bellman inequality implies
$$
\mathds{E}\sup_{0\leq t\leq T}(\vp_t^n(x))^2\leq C_1,
$$
where $C_1$ is a constant depending only on $T$ and $x$. This fact and monotonicity of the sequence $\{\tau_n: \ n\geq 1\}$ give
$\tau_n\to+\infty$ as $n\to\infty$ almost surely. Hence there exists a unique strong solution to equation (\ref{main_eq}).

To prove the differentiability of the flow let us consider an arbitrary interval $[x_1,x_2]$. By comparison theorem (cf. \cite{Nakao73}, Th. 2.1)
$\vp_t(x_1)\leq \vp_t(x)\leq\vp_t(x_2).$ Denote
$$
M_t=\max_{s\in[0,t]}\left(|\vp_s(x_1)|\vee|\vp_s(x_2)|\right).
$$
There exists $N>0$ such that $M_t<N.$ Then $\vp_s(x)=\vp_s^n(x)$ for all $x\in[x_1,x_2], \ s\in[0,t]$, and $n>N,$ almost surely. Consequently, for all $n>N$, the local times and the derivatives of the processes $\vp_s(x), \vp_s^n(x)$ coincide on $x\in[x_1,x_2], \ s\in[0,t]$.  This entails   assertions 2)-4) of the Theorem.
\end{proof}
\section{Stationary distribution}
Assume that a function $\aph$ satisfies conditions (\ref{BVonR}), (\ref{limONinfty}).
In this section we prove the existence of  a stationary distribution for the process
$(\vp_t(x))_{t\geq0}$ provided that conditions (\ref{BVonR}), (\ref{limONinfty}) are justified. Apply Theorem 3 of
\cite{Gikhman+68engl}, \S 18 to equation (\ref{main_eq}). Put
$$
s(x)=\int_0^x\exp\left\{-2\int_0^z\aph(y)dy\right\}dz, \ x\in\mathds{R}.
$$
By (\ref{limONinfty}),
$$
\begin{aligned}
s(x)\to +\infty,& \ \ x\to+\infty,\\
s(x)\to -\infty,& \ \ x\to-\infty.
\end{aligned}
$$
Besides, $s$ has a continuous positive derivative
$$
s'(x)=\exp\{-2\int_0^x\aph(z)dz\}, \ x\in\mathds{R}.
$$
Let $q(\cdot)=s^{-1}(\cdot)$ be a continuously differentiable function on $\mathds{R}$ inverse to $s(\cdot).$
The function $\eta_t(x)=s(\varphi_t(x))$
is a solution of the SDE
$$
\left\{
\begin{aligned}
d\eta_t(x)&=\sigma(\eta_t(x))dw(t),\\
\eta_0(x)&=s(x),
\end{aligned}
\right.
$$
where
$
\sigma(y)=s'(q(y))=\exp\{-2\int_0^{q(y)}\aph(z)dz\}, \ y\in\mathds{R}.
$
Using (\ref{limONinfty}) it is easy to see that
\begin{equation}\label{sigma_int}
\int_{-\infty}^{+\infty}\frac{1}{\sigma^2(y)}<\infty.
\end{equation}

The continuity of $q$ and boundedness of $\aph$ provide that the function $\sigma$ is locally Lipschitz continuous.  Let us see that $\sigma$ is globally Lipschitz continuous function on $\mathds{R}$. As a locally Lipschitz continuous function it has a derivative at almost all points $x\in\mathds{R},$ and the derivative is as follows
$$
\sigma'(y)=-2\aph(q(y))q'(y)\exp\left\{-\int_0^{q(y)}\aph(z)dz\right\}.
$$
Taking into account that
$$
q'(y)=\frac{1}{s'(q(y))}=\exp\left\{2\int_0^{q(y)}\aph(z)dz\right\},
$$
we arrive at the formula
$$
\sigma'(y)=-2\aph(q(y))
$$
valid for almost all $y\in\mathds{R}.$ Then according to the Newton-Leibniz formula for locally absolutely continuous functions, for all $\{x_1,x_2\}\subset\mathds{R},$
$$
|\sigma(x_2)-\sigma(x_1)|=\left|\int_{x_1}^{x_2}2\aph(q(y))dy\right|\leq 2||\aph||\cdot|x_2-x_1|.
$$
So $\sigma$ is Lipschitz continuous,
and the conditions of  \cite{Gikhman+68engl},\ \S 18, Theorem 3 are
fulfilled.
Let $\Phi_{t,x}(y)< \ y\in\mathds{R},$ be the distribution function of the random
variable $\varphi_t(x),$ i.e.
$$
\Phi_{t,x}(y)=P\{\vp_t(x)<y\}.
$$
The Theorem implies the  existence of a stationary distribution $P_{stat}(y), \ y\in\mathds{R}$, and for all $\{x,y\}\subset\mathds{R},$
$$
P_{stat}(y)=\lim_{t\to\infty}\Phi_{t,x}(y).
$$
\section{Proof of Theorem \ref{asymptot}}
Heuristically the asymptotic behavior of the difference  $\vp_t(x_2)-\vp_t(x_1)$ can be guessed as follows.
 If we represent the local time from (\ref{Newton_Leib}) by Tanaka's formula (\ref{local_phi}), then by the ergodic theorem, the last integral  in the right-hand side of (\ref{local_phi}) is equivalent to $t\int_{y}^{+\infty}\aph(z)dP_{stat}(z)$ as $t$ tends to $\infty$. The first member is bounded in probability because $\varphi_t(x)$ converges weakly to the stationary distribution. The stochastic integral in the right-hand side of  (\ref{local_phi}) is a continuous martingale with its characteristics being less than or equal to $t$. Therefore it is naturally to expect that $L_y^{\varphi(x)}(t)\sim t\int_{y}^{+\infty}\aph(z)dP_{stat}(z), \ t\to\infty,$ and, respectively,
$$
\ln\left(\varphi_t(x_2)-\varphi_t(x_1)\right)\sim t \int_{-\infty}^{+\infty}\left(-\int_z^{+\infty}\aph(y)dP_{stat}(y)\right)d\aph(z), \ t\to\infty.
$$
Below we give the rigorous proof of this fact.
\vskip 5 pt

\begin{proof}
In this proof we will use the representation of the function $\aph$ in the form
$$
\aph(x)=\aph_1(x)-\aph_2(x), \ x\in\mathds{R},
$$
where $\aph_1,\aph_2$ are nondecreasing functions on $\mathds{R}.$
Using Jensen's inequality we get the lower bound for $\frac{\ln(\varphi_t(x_2)-\varphi_t(x_1))}{t}$ as follows
\begin{multline}\label{lower bound}
\frac{\ln(\varphi_t(x_2)-\varphi_t(x_1))}{t}=\frac{\ln\left(\int_{x_1}^{x_2}\exp\left\{\int_{-\infty}^{+\infty}L_z^{\varphi(x)}(t)d\aph(z)\right\}dx\right)}{t}\\
\geq\frac{\int_{x_1}^{x_2}\ln\left((x_2-x_1)\exp\left\{\int_{-\infty}^{+\infty}L_z^{\varphi(x)}(t)d\aph(z)\right\}\right)dx}{t(x_2-x_1)}\\
=\frac{1}{t}\ln(x_2-x_1)
+\frac{1}{t}\int_{x_1}^{x_2}\frac{\int_{-\infty}^{+\infty}L_z^{\varphi(x)}(t)d\aph(z)}{x_2-x_1}dx.
\end{multline}

On the other hand, let $p_1,p_2, p_3$ be
grater than $1$ and such that $\sum_k \frac{1}{p_k}=1$. Then by H\"older's inequality we obtain
\begin{multline}\label{sum_of_f}
\frac{\ln(\varphi_t(x_2)-\varphi_t(x_1))}{t}=\frac{\ln\left(\int_{x_1}^{x_2}\exp\left\{\int_{-\infty}^{+\infty}L_z^{\varphi(x)}(t)d\aph(z)\right\}dx\right)}{t}\\
=\frac{\ln\left(\int_{x_1}^{x_2}\exp\left\{\sum_{k=1}^3 \int_{-\infty}^{+\infty}f_k(t,x,z)d\aph(z)\right\}dx\right)}{t}\\
\leq\frac{\ln\left(\prod_{k=1}^3\left(\int_{x_1}^{x_2}
\exp\left\{p_k\int_{-\infty}^{+\infty}f_k(t,x,z)d\aph(z)\right\}dx\right)^{1/p_k}\right)}{t}\\
=\frac{\sum_{k=1}^{3}\frac{1}{p_k}\ln\left(\int_{x_1}^{x_2}\exp\left\{p_k\int_{-\infty}^{+\infty}f_k(t,x,z)d\aph(z)\right\}dx\right)}{t},
\end{multline}
where
$$
\begin{aligned}
f_1(t,x,z)&=(\varphi_t(x)-z)^+-(x-z)^+,\\
f_2(t,x,z)&=-\int_0^t\mathds{1}_{(z,\infty)}(\varphi_s(x))dw(s),\\
f_3(t,x,z)&=-\int_0^t\mathds{1}_{(z,\infty)}(\varphi_s(x))\aph(\varphi_s(x))ds.\\
\end{aligned}
$$

Let us show that the right-hand side of (\ref{sum_of_f}) converges to \break $\int_{-\infty}^{+\infty}\left(-\int_z^{+\infty}\aph(y)dP_{stat}(y)\right)d\aph(z)$ almost surely. The same relation for the right-hand side of (\ref{lower bound}) can be proved similarly.

Consider the summand with $f_1(t,x,z)$.
It is easy to see that for all $\{x,z\}\subset\mathds{R}, \ t\geq0$,
$$
|(\varphi_t(x)-z)^+-(x-z)^+|\leq|\varphi_t(x)-x|.
$$
By the comparison theorem (cf. \cite{Nakao73}, Th. 2.1) for all $x\in[x_1,x_2]$,
$$
|\varphi_t(x)-x|\leq|\varphi_t(x_2)-x_1|.
$$
Then
\begin{multline}\label{f_1}
\frac{\ln\left(\int_{x_1}^{x_2}\exp{p_1\int_{-\infty}^{+\infty}f_1(t,x,z)d\aph(z)}dx\right)}{t}\\
\leq \frac{\ln\left(\int_{x_1}^{x_2}\exp\left\{p_1|\varphi_t(x_2)-x_1|\cdot\Var\aph\right\}dx\right)}{t}\\
\leq
\frac{\ln\left((x_2-x_1)\Var\aph\right)}{t}+\frac{p_1|\vp_t(x_2)|}{t}+\frac{p_1|x_1|}{t}.
\end{multline}

The first and the third summands obviously tend to $0$ as $t\to \infty$. Let us show that the same assertion is true for the second summand.

Put $c_1=a/2, \ c_2=b/2.$ Fix $\varepsilon\in\left(0,1/2\min(-a,b)\right)$.  There exist  $N_1>0$ and
$N_2<0$ such that
\begin{equation}
\begin{aligned}\label{aph_N}
\aph(x)&<a+\varepsilon<c_1, \ x\geq N_1,\\
\aph(x)&>b-\varepsilon>c_2, \ x\leq N_2.
\end{aligned}
\end{equation}
Consider the following stochastic differential equations
\begin{eqnarray}
\chi_t^1(x)&=&x+c_1t+w(t)+L_{N_1+1}^{\chi^1(x)}(t),\label{chi_1}\\
\chi_t^2(x)&=&x+c_2t+w(t)-L_{N_2-1}^{\chi^1(x)}(t),\label{chi_2}
\end{eqnarray}
where $\left(L_{N_1+1}^{\chi^1(x)}(t)\right)_{t\geq0},$
$\left(L_{N_2-1}^{\chi^2(x)}(t)\right)_{t\geq0}$ are local times
of the processes $\left(\chi_t^1(x)\right)_{t\geq0}$,
$\left(\chi_t^2(x)\right)_{t\geq0}$ at the points $N_1+1, N_2-1$
respectively.

There exist solutions of these equations (see (\cite{Lions+84})).
Starting from $x>N_1+1,$ the solution to the former equation is a diffusion process taking values on $[N_1+1,+\infty)$ with instantaneous reflection at the point $N_1+1$. For $x<N_2-1$, the solution of the latter equation is a diffusion process taking values on $(-\infty, N_2-1]$ with instantaneous reflection at the point $N_2-1$.

Given $x> N_1+1$, then
$$
P\{\varphi_t(x)\leq\chi_t^1(x), \ t\geq0\}=1.
$$
Indeed, let $t_{N_1+1}=\inf\{t:\chi_t^1(x)=N_1+1\}$. Then for all
$t\in(0,t_{N_1+1})$, by (\ref{aph_N})
$$
\chi_t^1(t)-\varphi_t(x)=\int_0^t(c_1-\aph(\varphi_s(x)))ds>0.
$$
Consequently, if there exists a
point $r_0
\geq t_{N_1+1}$ such that
$\chi_{r_0}^1(x)<\varphi_{r_0}(x)$, then there exists a point
$r_1\in[t_{N_1+1},r_0)$ at which
$\chi_{r_1}^1(x)=\varphi_{r_1}(x).$ Moreover $\vp_{r_1}(x)\geq N_1+1$. Choose $\delta>0$
such that for all $s\in[r_1,r_1+\delta]$, $\varphi_s(x)\geq N_1.$ Then
\begin{equation}\label{chi_s}
\chi_s^1(x)-\varphi_s(x)=\int_{r_1}^s(c_1-\aph(\varphi_s(x)))ds+L_{N_1+1}^{\chi^1(x)}(s)-L_{N_1+1}^{\chi^1(x)}(r_1), \ s\in[r_1,r_1+\delta].
\end{equation}
But the right-hand side of (\ref{chi_s}) is non-negative. This implies that for each
$x>N_1+1,$ and all $t\geq0,$ $\chi_t^1(x)\geq \varphi_t(x).$ By the
comparison theorem (see \cite{Nakao73}, Th. 3.1) $\chi_t^1(x)\leq
B_t^1(x), \ t\geq0,$ where $(B_t^1(x))_{t\geq0}$ is a
one-dimensional Brownian motion with reflection at the point
$N_1+1$, which is a solution to the following SDE
$$
B_t^1(x)=x+w(t)+L^{B(x)}_{N_1+1}.
$$
Thus for all $x>N_1+1$,
\begin{equation}\label{chi_B1}
\vp_t(x)\leq B_t^1(x), \ t\geq 0.
\end{equation}

Involving  $(\chi_t^2(x))_{t\geq0}$  and arguing in the same way we
get the inequality
\begin{equation}\label{chi_B2}
\vp_t(x)\geq B_t^2(x), \ t\geq 0,
\end{equation}
valid for all
$x<N_2-1$, where $(B_t^2(x))_{t\geq0}$ is a  Brownian motion with
reflection at the point $N_2-1$ solving the following SDE
\begin{equation*}
B_t^2(x)=x+w(t)-L^{B(x)}_{N_2-1}.
\end{equation*}
It is known that for all $x\in\mathds{R},$
\begin{equation}\label{B_converg}
\begin{aligned}
\frac{B_t^1(x)}{t}&\to 0,\ t\to\infty\\
\frac{B_t^2(x)}{t}&\to 0,\ t\to\infty.
\end{aligned}
\end{equation}
The fact that $\frac{\varphi_t(x)}{t}\to 0, \ t\to\infty,$ follows
now from relations (\ref{B_converg}), inequalities (\ref{chi_B1}), (\ref{chi_B2}) and assertion that for all $\{d_1,x,d_2\}\subset\mathds{R}, \ d_1<x<d_2$,
\begin{equation}\label{tau_infty}
\tau_x[d_1,d_2]<\infty \ \mbox{a.s.},
\end{equation}
where
$$
\tau_x[d_1,d_2]=\inf\{t\geq0:\vp_t(x)=d_1 \ \mbox{or} \
\vp_t(x)=d_2\}.
$$
Inequality (\ref{tau_infty}) is a consequence of (\ref{sigma_int}) (cf. \cite{Gikhman+68engl}, \S 18).

Thus we have proved that the second term in the right-hand side of (\ref{f_1}) tends to zero as $t$ tends to $\infty.$

Examine the third  item in the right-hand side of (\ref{sum_of_f}). We have
$$
\int_{-\infty}^{+\infty}\left(-\int_0^t\mathds{1}_{\varphi_s(x)>z}
\aph(\varphi_s(x))ds\right)d\aph(z)=\sum_{i,j=1}^2 \mathrm I_{ij},
$$
where
$$
\mathrm I_{ij}=(-1)^{i+j}\int_{-\infty}^{+\infty}\left(-\int_0^t
\mathds{1}_{\varphi_s(x)>z}\aph_i(\varphi_s(x))ds\right)d\aph_j(z).
$$
Consider $\mathrm I_{11}.$
By the comparison theorem for all $t\geq0,\ x\in[x_1,x_2],$
\begin{multline}\label{I_11}
\int_{-\infty}^{+\infty}\left(-\int_0^t\mathds{1}_{\varphi_s(x_2)>z}
\aph_1(\varphi_s(x_2))ds\right)d\aph_1(z)\leq \mathrm I_{11}\\
\leq \int_{-\infty}^{+\infty}\left(-\int_0^t\mathds{1}_{\varphi_s(x_1)>z}
\aph_1(\varphi_s(x_1))ds\right)d\aph_1(z).
\end{multline}
Using the similar estimates for $\mathrm I_{12}, \mathrm I_{21}, \mathrm I_{22}$ we get
\begin{multline}\label{ineq f 3}
\frac{1}{p_3t}\ln\left(\int_{x_1}^{x_2}\exp\left\{p_3\int_{-\infty}^{+\infty}f_3(t,x,z)d\aph(z)\right\}dx
\right)\geq\frac{1}{t}\left[\frac{\ln(x_2-x_1)}{p_3}\right.\\+
\int_{-\infty}^{+\infty}\left(-\int_0^t\mathds{1}_{\varphi_s(x_2)>z}
\aph_1(\varphi_s(x_2))ds\right)d\aph_1(z)\\
+\int_{-\infty}^{+\infty}\left(-\int_0^t\mathds{1}_{\varphi_s(x_1)>z}
\aph_2(\varphi_s(x_1))ds\right)d\aph_1(z)\\
+\int_{-\infty}^{+\infty}\left(-\int_0^t\mathds{1}_{\varphi_s(x_1)>z}
\aph_1(\varphi_s(x_1))ds\right)d\aph_2(z)\\
+\int_{-\infty}^{+\infty}\left(-\int_0^t\mathds{1}_{\varphi_s(x_2)>z}
\aph_2(\varphi_s(x_2))ds\right)d\aph_2(z)\left.\right].
\end{multline}
Obviously, the first summand in the right-hand side of (\ref{ineq f 3}) tends to $0$  as $t$ tends to $\infty.$
 By the ergodic theorem (see Theorem 3, \S 18 of \cite{Gikhman+68engl}) for all $x\in\mathds{R}, i=1,2,$ we get
$$
\frac{1}{t}\int_0^t\mathds{1}_{\vp_s(x)>z}\aph_i(\vp_s(x))ds\to \int_z^{+\infty}\aph_i(y)dP_{stat}(y).
$$
 Making use of the dominated convergence theorem and collecting the members, we see that the expression in the right-hand side of (\ref{ineq f 3}) tends to $$\int_{-\infty}^{+\infty}\left(-\int_z^{+\infty}\aph(y)dP_{stat}(y)\right)d\aph(z)$$ almost surely as $t$ tends to $\infty$. Using the upper estimates for $\mathrm I_{ij}, \ \{i,j\}\subset\{1,2\},$ similarly we get that
\begin{multline*}
\frac{1}{p_3t}\ln\left(\int_{x_1}^{x_2}\exp\left\{p_3\int_{-\infty}^{+\infty}f_3(t,x,z)d\aph(z)\right\}dx
\right)\to \\
\int_{-\infty}^{+\infty}\left(-\int_z^{+\infty}\aph(y)dP_{stat}(y)\right)d\aph(z), \
t\to\infty, \ \mbox{almost surely.}
\end{multline*}
It is left to prove that the second member in the right-hand side of
(\ref{sum_of_f}) converges to $0$ as $t$ tends to $\infty$ almost surely.
It can be represented in the form
\begin{multline*}
\frac{\ln\left(\int_{x_1}^{x_2}\exp\left\{p_2\int_{-\infty}^{+\infty}
\left(-\int_0^t\mathds{1}_{\varphi_s(x)>z}dw(s)
\right)d\aph(z)\right\}dx\right)}{t}\\
=\frac{\ln(x_2-x_1)}{t}+\frac{p_2\int_{-\infty}^{+\infty}
\left(-\int_0^t\mathds{1}_{\varphi_s(x_1)>z}dw(s)
\right)d\aph(z)}{t}\\
+\frac{\ln\left(\int_{x_1}^{x_2}\exp\left\{p_2\int_{-\infty}^{+\infty}
\left(-\int_0^t\left(\mathds{1}_{\varphi_s(x)>z}-\mathds{1}_{\varphi_s(x_1)>z}\right)dw(s)
\right)d\aph(z)\right\}dx\right)}{t}
=I+II+III.
\end{multline*}
Consider $II$. By a martingale inequality (cf. \cite{Ikeda+81}), ineq. (6.16) of Ch. 1),
\begin{multline*}
\mathds{E}\sup_{r\in[0,t]}\left(\int_{-\infty}^{+\infty}\left(-\int_0^r\mathds{1}_{\varphi_s(x_1)>z}dw(s)\right)d\aph(z)\right)^2\leq\\
=\mathds{E}\sup_{r\in[0,t]}\left(\int_0^\tau\left(\int_{-\infty}^{+\infty}\mathds{1}_{\varphi_s(x_1)>z}d\aph(z)\right)dw(s)\right)^2\leq\\
4\mathds{E}\int_0^t\left(\int_{-\infty}^{+\infty}\mathds{1}_{\varphi_s(x_1)>z}d\aph(z)\right)^2ds\leq 4(\Var_{\mathds{R}}\aph)^2t.
\end{multline*}
Then by monotone convergence theorem
\begin{multline*}
\mathds{E}\sum_{n=1}^\infty\sup_{r\in[2^n,2^{n+1}]}\left(\frac{\int_{-\infty}^{+\infty}
\left(-\int_0^r\mathds{1}_{\varphi_s(x_1)>z}dw(s)\right)d\aph(z)}{r}\right)^2\\
\leq
\sum_{n=1}^{\infty}\frac{\mathds{E}\sup_{r\in[0,2^{n+1}]}\left(\int_{-\infty}^{+\infty}\left(-\int_0^r\mathds{1}_{\varphi_s(x_1)>z}dw(s)\right)
d\aph(z)\right)^2}{2^{2n}}\\
\leq \left(\Var_{\mathds{R}}\aph\right)^2\sum_{n=1}^{\infty}\frac{4\cdot 2^{n+1}}{2^{2n}}=\left(\Var_\mathds{R}\aph\right)^2\sum_{n=1}^{\infty}\frac{8}{2^n}<\infty.
\end{multline*}
This implies that
$$
\sup_{\tau\in[2^n,2^{n+1}]}\left(\frac{\int_{-\infty}^{+\infty}\left(-\int_0^\tau\mathds{1}_{\varphi_s(x_1)>z}dw(s)\right)d\aph(z)}{r}\right)^2\to 0, \ n\to\infty, \ \mbox{almost surely}.
$$
Consequently,
\begin{equation}\label{eee}
\lim_{t\to\infty}\frac{\int_{-\infty}^{+\infty}\left(-\int_0^t\mathds{1}_{\varphi_s(x_1)>z}dw(s)\right)d\aph(z)}{t}=0 \ \mbox {almost surely.}
\end{equation}

Note that
\begin{equation}\label{f_3_2}
III\leq\frac{(x_2-x_1)}{t}-\frac{p_2\sup_{x\in[x_1,x_2]}\sup_{r\in[0,t]}\left|\int_0^r\left(\int_{-\infty}^{+\infty}
\left(\mathds{1}_{\varphi_s(x)>z}-\mathds{1}_{\varphi_s(x_1)>z}
\right)d\aph(z)\right)dw(s)\right|}{t}.
\end{equation}
To prove that $III\to0$ as $t\to\infty$ it is sufficient to show that
\begin{equation}\label{iiiii}
\frac{\sup_{x\in[x_1,x_2]}\sup_{r\in[0,t]}\left|\int_0^r\left(\int_{-\infty}^{+\infty}
\left(\mathds{1}_{\varphi_s(x)>z}-\mathds{1}_{\varphi_s(x_1)>z}
\right)d\aph(z)\right)dw(s)\right|}{t}\to 0, \ t\to\infty, \ \mbox{almost surely}.
\end{equation}

Put $\xi_t(x)=\int_0^t\left(\int_{-\infty}^{+\infty}\mathds{1}_{\varphi_s(x)>z}d\aph(z)\right)dw(s).$
According to the Garsia-Rodemich-Rumsey inequality \cite{Garsia+70} for all $t\geq0,\ x\in[x_1,x_2]$, $q>1,\
\alpha\in(\frac{1}{q},1],$ there exists $c(\alpha,q)>0$ such that
\begin{equation}\label{Garcia_ineq}
|\xi_t(x)-\xi_t(x_1)|^q\leq
c(\alpha,q)|x-x_1|^{q\alpha-1}\int\!\!\!\!\int_{[x_1,x_2]^2}\frac{|\xi_t(u)-\xi_t(v)|^q}{|u-v|^{q\alpha+1}}dudv.
\end{equation}
Then for $q=4,$
\begin{multline}\label{aaa}
\mathds{E}\sup_{x\in[x_1,x_2]}\sup_{t\in[0,T]}|\xi_t(x)-\xi_t(x_1)|^4\\
\leq
c(\alpha,4)|x_2-x_1|^{4\alpha-1}\int\!\!\!\!\int_{[x_1,x_2]^2}\frac{\mathds{E}\sup_{t\in[0,T]}|\xi_t(u)-\xi_t(v)|^4}{|u-v|^{4\alpha+1}}dudv\\
\end{multline}
Let us estimate the expectation $\mathds{E}\sup_{t\in[0,T]}|\xi_t(u)-\xi_t(v)|^4$.
According to Burkholder's inequality (cf. \cite{Ikeda+81}, Ch.3, Th. 3.1) we get
\begin{multline*}
\mathds{E}\sup_{t\in[0,T]}|\xi_t(u)-\xi_t(v)|^4\\
\leq C\mathds{E}\left(\int_0^T\left(\int_{-\infty}^{+\infty}\left(\mathds{1}_{\varphi_s(u)>z}-\mathds{1}_{\varphi_s(v)>z}\right)d\aph(z)\right)^2ds\right)^2.
\end{multline*}
Consider the case of $u<v$. Making use of H\"{o}lder's inequality and applying the comparison theorem
we arrive at the inequality
\begin{multline*}
\left(\int_{-\infty}^{+\infty}\left(\mathds{1}_{\varphi_s(u)>z}-\mathds{1}_{\varphi_s(v)>z}\right)d\aph(z)\right)^2\\
\leq
2\Var\aph_1\int_{-\infty}^{+\infty}\left(\mathds{1}_{\varphi_s(u)>z}-\mathds{1}_{\varphi_s(v)>z}\right)^2d\aph_1(z)\\
+2\Var\aph_2\int_{-\infty}^{+\infty}\left(\mathds{1}_{\varphi_s(u)>z}-\mathds{1}_{\varphi_s(v)>z}\right)^2d\aph_2(z)\\
\leq C\int_{-\infty}^{+\infty}\left(\mathds{1}_{\varphi_s(u)>z}-\mathds{1}_{\varphi_s(v)>z}\right)d\bar{\aph}(z),
\end{multline*}
where for $z\in\mathds{R}$, $\bar{\aph}(z)=\aph_1(z)+\aph_2(z)$,  $C$ is a constant.

Then
\begin{multline*}
\mathds{E}\sup_{t\in[0,T]}|\xi_t(u)-\xi_t(v)|^4
\leq C\mathds{E}\left(\int_0^T\left(\int_{-\infty}^{+\infty}\left(\mathds{1}_{\varphi_s(u)>z}-\mathds{1}_{\varphi_s(v)>z}\right)d\bar{\aph}(z)\right)ds\right)^2\\
=C\int_\mathds{R}d\bar{\aph}(z)\int_\mathds{R}d\bar{\aph}(y)\mathds{E}\left[\int_0^T\left(
\left(\mathds{1}_{\varphi_s(u)>z}-\mathds{1}_{\varphi_s(v)>z}\right)\int_s^T\left(\mathds{1}_{\varphi_r(u)>y}-\mathds{1}_{\varphi_r(v)>y}\right)dr\right)ds\right]\\
\leq
 C\int_\mathds{R}d\bar{\aph}(z)\int_\mathds{R}d\bar{\aph}(y)\\
\times\mathds{E}\left[\int_0^T\left(
\left(\mathds{1}_{\varphi_s(u)>z}-\mathds{1}_{\varphi_s(v)>z}\right)
\mathds{E}\left(\int_s^T\left(\mathds{1}_{\varphi_r(u)>y}-
\mathds{1}_{\varphi_r(v)>y}\right)dr\slash\mathfrak{F}_s\right)\right)ds\right]\\
\leq
 C\int_\mathds{R}d\bar{\aph}(z)\int_\mathds{R}d\bar{\aph}(y)\\
 \times\mathds{E}\left[\int_0^T\left(
\left(\mathds{1}_{\varphi_s(u)>z}-\mathds{1}_{\varphi_s(v)>z}\right)
\mathds{E}\left(\int_0^{T-s}\left(\mathds{1}_{\varphi_r(u)>y}-
\mathds{1}_{\varphi_r(v)>y}\right)dr\right)\right)ds\right].
\end{multline*}
valid for all $T>0, \ u\in\mathds{R}, \ v\in\mathds{R}, \ u<v,$
with some constant $C$.

By arguments similar to that in  \cite{Gikhman+68engl}, \S 18, Remark 1 we have
$$
\mathds{E}\int_0^T(\mathds{1}_{\varphi_s(u)>z}-\mathds{1}_{\varphi_s(v)>z})ds\leq H(u-v), \ z\in\mathds{R},
$$
where $H$ is some positive constant. This implies
$$
\mathds{E}\sup_{x\in[x_1,x_2]}\sup_{t\in[0,T]}|\xi_t(x)-\xi_t(x_1)|^4\leq (\Var\bar{\aph})^2 H^2(u-v)^2.
$$
The case of $u\geq v$ can be treated analogously. Thus the inequality
\begin{multline}\label{aaa}
\mathds{E}\sup_{x\in[x_1,x_2]}\sup_{t\in[0,T]}|\xi_t(x)-\xi_t(x_1)|^4\\
\leq
c(\alpha,4)|x_2-x_1|^{4\alpha-1}\int\!\!\!\!\int_{[x_1,x_2]^2}\frac{C(u-v)^2}{|u-v|^{4\alpha+1}}dudv
\end{multline}
holds true for all $T>0, \{x_1,x_2\}\subset\mathds{R}, \ x_1<x_2.$
To provide the finiteness of the integral in the right-hand side
of (\ref{aaa}) we choose $\alpha$ such that $1-4\alpha>-1$, i.e.
$\alpha\in(\frac14,\frac12).$ Finally, calculating the integral we
get
\begin{equation*}
\mathds{E}\sup_{x\in[x_1,x_2]}\sup_{t\in[0,T]}|\xi_t(x)-\xi_t(x_1)|^4
\leq
C(x_2-x_1)^2,
\end{equation*}
where $C$ is a constant.

This inequality implies that the convergence in (\ref{iiiii}) holds in probability.
The almost surely convergence can be justified by arguments similar to that used in the proof  of formula (\ref{eee}). So we checked that $III\to 0$ as $t\to\infty$. This completes the proof of the fact that
$$
\lim_{t\to\infty}\frac{\ln(\vp_t(x_2)-\vp_t(x_1))}{t}\leq \int_{-\infty}^{+\infty}\left(-\int_z^{+\infty}\aph(y)dP_{stat}(y)\right)d\aph(z) \ \ \mbox{(see (\ref{sum_of_f}))}.
$$
Treating (\ref{lower bound}) analogously we get
$$
\int_{-\infty}^{+\infty}\left(-\int_z^{+\infty}\aph(y)dP_{stat}(y)\right)d\aph(z)\\
\leq\lim_{t\to\infty}\frac{\ln(\vp_t(x_2)-\vp_t(x_1))}{t}
$$
The Theorem 2 is proved.
\end{proof}

\section{Example}
Let  $\aph(x)=a\mathds{1}_{x\geq0}+b\mathds{1}_{x<0},$ where $a<0,
b>0.$ Given $\{x_1,x_2\}\subset\mathds{R},$ the processes
{$(\vp_t(x_1))_{t\geq0}$, $(\vp_t(x_2))_{t\geq0}$} move parallel
to each other while being on the same semiaxis.
 Theorem \ref{meeting} holds true  for the solution $(\vp_t(x))_{t\geq0}$ of corresponding SDE. The Sobolev derivative has the form (see (\ref{Sobolev_der_expres}))
$$
\nabla\varphi_t(x)=\exp\left\{(a-b)L_0^{\varphi(x)}(t)\right\},
$$
where $L_0^{\varphi(x)}(t)$ is a local time of the process
$(\varphi_t(x))_{t\geq0}$ at the point zero.

Let us find stationary distribution for the process $(\vp_t(x))_{t\geq0}$ (see Section 6).
We have
$$
s(x)=
\begin{cases}-\frac{1}{2a}\left(e^{-2ax}-1\right), &  x\geq0,\\
-\frac{1}{2b}\left(e^{-2bx}-1\right), &  x<0,
\end{cases}
$$
$$
s'(x)=
\begin{cases}e^{-2ax}, & \ x\geq0,\\
e^{-2bx}, & \ x<0,
\end{cases}
$$
and
$$
\sigma(y)=s'(q(y))=\begin{cases}
-1-2ay,& y\geq0,\\
-1-2by,& y<0,
\end{cases}
$$
where
$$
q(y)=\begin{cases}
-\frac{1}{2a}\ln(1-2ay),& y\geq0,\\
-\frac{1}{2b}\ln(1-2by),& y<0,
\end{cases}
$$
 is a continuously differentiable inverse function to $s(\cdot).$

Put $\eta_t(x)=s(\varphi_t(x)).$
Then (see Section 6) it is a solution of the SDE
$$
\left\{
\begin{aligned}
d\eta_t(x)&=\sigma(\eta_t(x))dw(t),\\
\eta_0(x)&=s(x).
\end{aligned}
\right.
$$
Let $F_{t,x}(y)=P\{\eta_t(x)<y\}$ be the distribution
function of the random variable $\eta_t(x)$. Then by  Theorem 3 of  \cite{Gikhman+68engl},\ \S 18, for all $x\in\mathds{R},$
\begin{equation}\label{Distr_eta}
\lim_{t\to\infty}F_{t,x}(y)=\frac{\int_{-\infty}^y\frac{dz}{\sigma^2(z)}}{\int_{-\infty}^{+\infty}\frac{dz}{\sigma^2(z)}}=
\begin{cases}
1+\frac{b}{a-b}\frac{1}{1-2ay},& y\geq0,\\
\frac{a}{a-b}\frac{1}{1-2by},& y<0.
\end{cases}
\end{equation}

Let $\Phi_{t,x}(y), \ y\in\mathds{R},$ be the distribution function of the random
variable $\varphi_t(x).$ From (\ref{Distr_eta}) for all $x\in\mathds{R}$,  we have
\begin{multline*}\label{Distr_phi}
P_{stat}(y)=\lim_{t\to\infty}\Phi_{t,x}(y)=\lim_{t\to\infty}P\{\varphi_t(x)<y\}=\lim_{t\to\infty}P\{\eta_t(x)<s(y)\}\\
=\lim_{t\to\infty}F_{t,x}(s(y))=
\begin{cases}
1+\frac{b}{a-b}e^{2ay},& y\geq0,\\
\frac{a}{a-b}e^{2by},& y<0.
\end{cases}
\end{multline*}
The stationary distribution function $P_{stat}(y)$ has a density of the form
\begin{equation}\label{density_phi}
p_{stat}(y)=
\begin{cases}
\frac{2ab}{a-b}e^{2ay},& y\geq0,\\
\frac{2ab}{a-b}e^{2by},& y<0.
\end{cases}
\end{equation}

Theorem 2 now is as follows. For all
$\{x_1,x_2\}\subset\mathds{R}, \ x_1<x_2,$
\begin{equation*}
\lim_{t\to+\infty}\frac{\ln(\varphi_t(x_2)-\varphi_t(x_1))}{t}=(b-a)\int_0^{+\infty}\frac{2a^2b}{a-b}e^{2ay}\mathds{1}_{y\geq0}dy=ab
\ \mbox{almost surely}.
\end{equation*}

\bibliographystyle{plain}


\begin{thebibliography}{10}

\bibitem{Zvonkin74}
A.K.Zvonkin.
\newblock A transformation of the phase space of a~diffusion process that
  removes the drift.
\newblock {\em Mat. Sb. (N.S.)}, 93(135):129--149, 1974.

\bibitem{Attanasio10}
S.~Attanasio.
\newblock Stochastic flows of diffeomorphisms for one-dimensional sde with
  discontinuous drift.
\newblock {\em Electron. Commun. Probab.}, 15:no. 20, 213--226, 2010.

\bibitem{Barlow+01}
M.T. Barlow, K.~Burdzy, H.~Kaspi, and A.~Mandelbaum.
\newblock Coalescence of skew {Brownian} motions.
\newblock {\em Seminaire de probabilites, Universite de Strasbourg},
  XXXV:202--205, 2001.

\bibitem{Bouleau+91}
N.~Bouleau and F.~Hirsch.
\newblock {\em Dirichlet forms and analysis on Wiener space}.
\newblock De Gruyter studies in mathematics. W. de Gruyter, 1991.

\bibitem{Burdzy+04}
K.~Burdzy and H.~Kaspi.
\newblock Lenses in skew {B}rownian flow.
\newblock {\em The Annals of Probability}, 32:3085--3115, 2004.

\bibitem{Flandoli+10}
F.~Flandoli, M.~Gubinelli, and E.~Priola.
\newblock Flow of diffeomorphisms for sdes with unbounded hölder continuous
  drift.
\newblock {\em Bulletin des Sciences Mathematiques}, 134(4):405 -- 422, 2010.

\bibitem{Garsia+70}
A.M. Garsia, E.~Rodemich, and H.jun. Rumsey.
\newblock {A real variable lemma and the continuity of paths of some Gaussian
  processes.}
\newblock {\em Math. J., Indiana Univ.}, 20:565--578, 1970.

\bibitem{Harrison+81}
J.~M. Harrison and L.~A. Shepp.
\newblock On skew brownian motion.
\newblock {\em Ann. Probab.}, 9(2):9--13, 1981.

\bibitem{Gikhman+68engl}
I.I.Gikhman and A.V.Skorokhod.
\newblock {\em Stochastic differential equations}.
\newblock Naukova Dumka, Kiev, 1968.
\newblock [Translated from the Russian to the English by K. Wickwire.
  Springer-Verlag, Heidelberg-New York, 1972].

\bibitem{Ikeda+81}
N.~Ikeda and S.~Watanabe.
\newblock {\em Stochastic differential equations and diffusion processes}.
\newblock Kodansha ltd., Tokyo, 1981.

\bibitem{Kulik+00}
A.~M. Kulik and A.~Yu. Pilipenko.
\newblock Nonlinear transformations of smooth measures on infinite-dimensional
  spaces.
\newblock {\em Ukrainian Mathematical Journal}, 52:1403--1431, 2000.
\newblock 10.1023/A:1010380119199.

\bibitem{Kunita90}
H.~Kunita.
\newblock {\em Stochastic Flows and Stochastic Differential Equations}.
\newblock Cambridge Univ. Press, 1990.

\bibitem{Lions+84}
P.~L. Lions and A.~S. Sznitman.
\newblock Stochastic differential equations with reflecting boundary
  conditions.
\newblock {\em Communications on Pure and Applied Mathematics}, 37(4):511--537,
  1984.

\bibitem{McKean69}
H.~P. McKean.
\newblock {\em Stochastic integrals}.
\newblock Academic Press, New York, London, 1969.

\bibitem{Nakao73}
S.~Nakao.
\newblock Comparison theorems for solutions of one-dimensional stochastic
  differential equations.
\newblock In G.~Maruyama and Yu. Prokhorov, editors, {\em Proceedings of the
  Second Japan-USSR Symposium on Probability Theory}, volume 330 of {\em
  Lecture Notes in Mathematics}, pages 310--315. Springer Berlin / Heidelberg,
  1973.
\newblock 10.1007/BFb0061496.

\bibitem{Nikolsliy75}
S.M. Nikolsky.
\newblock {\em A Course of Mathematical Analysis (Vol. 2)}.
\newblock Nauka, Moscow, 2 edition, 1975.
\newblock [Translated from the Russian to the English by V.M. Volosov. Mir
  Pub., Moscow, 1977].

\bibitem{Billingsley68}
P.Billingsley.
\newblock {\em Convergence of probability measures}.
\newblock J.Wiley \& Sons, New York, 1968.

\bibitem{Portenko90}
N.I. Portenko.
\newblock {\em Generalized diffusion processes}.
\newblock Translations of mathematical monographs. American Mathematical
  Society, 1990.

\bibitem{Revuz+99}
D.~Revuz and M.~Yor.
\newblock {\em Continuous martingales and Brownian motion}.
\newblock Springer-Verlag, Berlin, 1999.

\end{thebibliography}

\end{document}